\documentclass
{article}
\usepackage[utf8]{inputenc}

\usepackage{amsmath}
\usepackage{amsfonts}
\usepackage{amssymb}

\usepackage{amsmath, amsthm, amssymb, amsfonts}
\usepackage{graphicx, color} 
\usepackage{fancyhdr} 
\usepackage{caption, subcaption} 
\usepackage{enumerate} 
\usepackage{mathrsfs} 
\usepackage{makeidx} 
\usepackage[section]{placeins} 
\usepackage[all]{xy} 
\usepackage{appendix}
\usepackage{verbatim}
\usepackage{moresize}

\usepackage[mathlines, pagewise]{lineno}
\usepackage
{hyperref}
\usepackage{enumitem}
\hypersetup{
  colorlinks  = true,    
  urlcolor    = red,    
  linkcolor   = blue,    
  citecolor   = blue,      
  pdfauthor   = {Hector Alonzo Barriga Acosta},%
  pdfsubject  = {Protocolo},%
  pdftitle    = {Countable Nabla Products},  %
}

\usepackage[capitalise, noabbrev,nameinlink]{cleveref}

\newtheorem{prop}{Proposition}
\newtheorem{thr}[prop]{Theorem}
\newtheorem{coro}[prop]{Corollary}
\newtheorem{lemma}[prop]{Lemma}
\newtheorem{de}[prop]{Definition}
\newtheorem{exa}[prop]{Example}

\newtheorem{ques}[prop]{Question}

\newtheorem{prob}[prop]{Problem}

\newcommand{\om}{\omega}
\newcommand{\baire}{\omega^\omega}

\newcommand{\sub}{\subseteq}

\newcommand{\nablaomega}{\nabla (\om +1)^\om}
\newcommand{\boxomega}{\square (\omega +1)^\omega}

\newcommand{\topl}{\overline{c}_\om}

\newcommand{\propp}{\mathcal{P}}

\newcommand{\rest}{\upharpoonright}

\newcommand{\xbar}{\overline{x}}
\newcommand{\ybar}{\overline{y}}

\newcommand{\zbar}{\overline{z}}

\newcommand{\down}[1]{\downarrow{#1}}

\newcommand{\ordnabla}[1]{\nabla ({#1} +1)^\om}
\newcommand*\rfrac[2]{{}^{#1}\!/_{#2}} 

\usepackage{cite}

\title{Monotone Normality and Nabla-Products}
\author{H.A. Barriga-Acosta and P.M. Gartside 
}
\date{}
\begin{document}
\maketitle

\begin{abstract}
Roitman's combinatorial principle $\Delta$ is equivalent to  monotone normality of the nabla product, $\nabla (\om +1)^\om$. If $\{ X_n : n\in \om\}$ is a family of metrizable spaces and $\nabla_n X_n$ is monotonically normal, then $\nabla_n X_n$ is hereditarily paracompact. 
Hence, if $\Delta$ holds then the box product $\square  (\om +1)^\om$ is paracompact.
Large fragments of $\Delta$ hold in $\mathsf{ZFC}$, yielding large subspaces of $\nabla (\om+1)^\om$ that are `really' monotonically normal.
Countable nabla products of metrizable spaces which are respectively: arbitrary,  of size $\le \mathfrak{c}$, or separable,  are monotonically normal under respectively: $\mathfrak{b}=\mathfrak{d}$, $\mathfrak{d}=\mathfrak{c}$ or the Model Hypothesis.

It is consistent and independent that $\nabla A(\omega_1)^\omega$ and $\nabla (\om_1+1)^\om$
are  hereditarily normal (or hereditarily paracompact, or monotonically normal). 
In $\mathsf{ZFC}$ neither $\nabla A(\omega_2)^\omega$ nor $\nabla (\om_2+1)^\om$ is hereditarily normal.
\end{abstract}

\section{Introduction}
Let $\{ X_i : i\in I  \}$ be a family of topological spaces. (All spaces in this article are Tychonoff.) A \emph{box} is a set $\prod_i U_i$, where each $U_i$ is open in $X_i$. The \emph{box product}, $\square_i X_i$, is the space with underlying set $\prod_i X_i$ and basis all boxes. Two elements $x$ and $y$ of $\square_i X_i$ are \emph{mod-finite equivalent}, denoted $x \sim y$, if the set $\{i \in I : x(i) \ne y(i)\}$ is finite. The \emph{nabla product}, $\nabla_i X_i$, is the quotient space, $\square_i X_i/\sim$.

It is unknown, in $\mathsf{ZFC}$,  whether the countable box product $\square [0,1]^\om$, or even its closed subspace, $\square (\omega+1)^\om$, is normal. This question was asked (orally) for the first time by Tietze sometime in the 1940's. See \cite{roitman2015paracompactness} for a survey of the box product problem. 
Central to almost all positive results on paracompactness, and hence normality, of box products, is a connection to the nabla product due to Kunen \cite{kunen1978paracompactness}: let $\{ X_n : n\in \om \}$ be a family of compact spaces, then, $\square_n X_n$ is paracompact if and only if $\nabla_n X_n$ is paracompact. 
In particular, it is now known that under certain set theoretic assumptions  the nabla product $\nabla (\om +1)^\om$ is paracompact and so the box product $\square (\om +1)^\om$ is paracompact.
These assumptions include the small cardinal conditions, $\mathfrak{b} =\mathfrak{d}$ \cite{van1980covering}, and $\mathfrak{d}= \mathfrak{c}$ \cite{roitman1979more, lawrence1988box}, 
and also the so called Model Hypothesis \cite{roitman2011paracompactness}, which holds in any forcing extension by uncountably many Cohen reals.  

In an insightful analysis of the combinatorics behind these consistency results,   Roitman \cite{roitman2011paracompactness} extracted  a combinatorial principle,  $\Delta$. She showed $\Delta$    is a consequence of each of the set theoretic axioms mentioned above, and further claimed that $\Delta$ implies the paracompactnes of $\nabla (\om +1)^\om$.
Unfortunately not all the details for the latter deduction were presented, and the authors and Roitman \cite{R_personal} are unclear how to fill the gap. See \cref{Delta} for the definition of $\Delta$, additional notation and more details on the gap.

In \cref{S:Delta_and_Nabla}  we close the gap by connecting $\Delta$ to \emph{monotone} normality of $\nabla$ products.
Indeed, (\cref{MN}) the property $\Delta$ holds if and only if $\nabla (\om +1)^\om$ is monotonically normal.
Monotonically normal spaces are not automatically paracompact, but (\cref{ThrNablaMNimpliesHP}) we show: if $\{ X_n : n\in \om \}$ is a family of metrizable spaces and $\nabla_n X_n$ is monotonically normal, then it is hereditarily paracompact. 
It follows that if $\Delta$ holds, then $\nabla (\om +1)^\om$ is monotonically normal, and so hereditarily paracompact, and hence $\square  (\om+1)^\om$ is paracompact, as Roitman originally claimed.

Recall that a space $X$  is \emph{monotonically normal} if for every pair of disjoint closed sets $A, B$ there is an open set $H(A,B)$ such that (i) $A \subseteq H(A,B) \subseteq \overline{H(A,B)} \subseteq X \setminus B$ (so $H(\cdot,\cdot)$ separates $A$ from $B$, and thus  witnesses normality) and (ii) if $A' \subseteq A$ and $B \subseteq B'$ then $H(A',B') \subseteq H(A,B)$ (`monotonicity', the separation respects set inclusion). 
An alternative characterization is that for every point $x$ in an open set $U$ there is assigned an open set $G(x,U)$  such that $x\in G(x, U) \subseteq U$, and if $G(x, U) \cap G(y, V) \neq \emptyset$, then $x \in V$ or $y \in U$. 
Observe that the restriction of a monotone normality operator, $G(\cdot,\cdot)$, to a subspace yields a monotone normality operator for the subspace, and so monotone normality is hereditary. 
It follows that monotone normality does not transfer from $\nabla (\om +1)^\om$ to $\square (\om +1)^\om$.
Indeed, see \cite{williams1984box}\label{bp not hn}, 
if $\{ X_i : i\in I \}$ is a family of compact or first countable spaces, then $\square_i X_i$ is not hereditarily normal.

The authors do not know how to prove $\Delta$ in $\mathsf{ZFC}$, or to prove that its negation is consistent. 
In an effort to shed light on this conundrum we have attempted to `parametrize' the problem: either `from below' in order to see how close we can get to establishing $\Delta$ in $\mathsf{ZFC}$, or `from above' to determine when natural strengthenings of $\Delta$ are false either consistently or in $\mathsf{ZFC}$.

For example, pursuing an idea of Roitman, we characterize in \cref{MNNabla} when certain subspaces $A$ of $\nabla (\om+1)^\om$ are monotonically normal in terms of a combinatorial property $\Delta (A)$, where $\Delta$ is $\Delta(\nabla (\om+1)^\om)$. In particular, see \cref{P:DeltaFI},  $\Delta (A)$ is true in $\mathsf{ZFC}$ for $A$ consisting of all finite disjoint unions of increasing functions.

In the other direction we have found combinatorial characterizations of when nabla products of certain spaces containing $\om+1$ as a closed subspace, or otherwise naturally extending $\om+1$, are monotonically normal. 
Specifically, observe that $\omega+1$ is the one-point compactification of a countably infinite discrete space. 
Accordingly in \cref{nabla_of_fort}, we characterize combinatorially monotone normality of nabla products of spaces of the form, $A(\kappa)$, the one point compactification of a discrete space of size $\kappa$.
We  show that $\nabla A(\om_2)^\om$ is not hereditarily normal, and so not monotonically normal, in $\mathsf{ZFC}$; while $\nabla A(\om_1)^\om$ is consistently not hereditarily normal. A striking  result of Williams \cite{williams1984box}, is that consistently any countable nabla product of compact spaces of weight (minimal size of a base) no more than $\aleph_1$ is ($\om_1$-metrizable and so) monotonically normal and hereditarily paracompact. 
In particular, $\nabla A(\om_1)^\om$ is consistently monotonically normal and hereditarily paracompact; and so each of the statements: `$\nabla A(\om_1)^\om$ is monotonically normal', `$\nabla A(\om_1)^\om$ is hereditarily paracompact' and `$\nabla A(\om_1)^\om$ is hereditarily  normal' is consistent and independent. These results answer questions of Roitman. 

Observing that $\omega+1$ can also be viewed as an ordinal with the order topology, in \cref{nabla_of_ordinals}, we go on to characterize combinatorially monotone normality of nabla products of  ordinals. 
This yields a finer parametrization than looking at one point compactifications. 
Indeed  if $\nabla (\om+1)^\om$ is monotonically normal (in other words, $\Delta$ holds) then for every $n$ in $\om$, we also have $\nabla (\om.n+1)^\om$ monotonically normal. 
However,  the combinatorial principle characterizing when $\nabla (\om.\om+1)^\om$ is monotonically normal is -- at least \emph{formally} -- stronger than $\Delta$. 
 By the result of Williams \cite{williams1984box} mentioned above, $\nabla (\om_1+1)^\om$ is consistently monotonically normal and hereditarily paracompact. 
 We show $\nabla (\om_1+1)^\om$ is consistently not hereditarily normal.
 Thus each of the statements: `$\nabla (\om_1+1)^\om$ is monotonically normal', `$\nabla (\om_1+1)^\om$ is hereditarily paracompact' and `$\nabla (\om_1+1)^\om$ is hereditarily  normal' is consistent and independent.  
In contrast we also show that $\nabla (\om_2+1)^\om$ is not hereditarily normal, and so not monotonically normal, in $\mathsf{ZFC}$. 
These results answer questions of Roitman.

While in \cref{nabla_of_metrizable} -- thinking of $\omega+1$ as the simplest non-discrete metrizable space -- we investigate combinatorial characterizations of monotone normality of nabla products of metrizable spaces. 
We conclude the paper with some open problems and potential lines of research.


\section{Preliminaries}

\subsection{Set Theory} 

Recall that $\mathfrak{b}$ is the minimal size of an unbounded set in $\om^\om$ with the mod-finite order, $\leq^*$, and $\mathfrak{d}$ the minimal size of a  cofinal (dominating) set. 
Further, $\mathfrak{b}=\mathfrak{d}$ if and only if there is a dominating family $\{f_\alpha : \alpha < \mathfrak{b}\} \sub \baire$ so that if $\alpha < \beta$ then $f_\alpha \leq^* f_\beta$ (such a family is called a \emph{scale}). We record an additional useful fact, a proof of which can be found in \cite{roitman2011paracompactness}.

\begin{lemma}\label{LemmaDiagonilizingFunctions}
If $\mathcal{G}\subseteq \baire, \mathcal{A} \subseteq \mathcal{P}(\omega)$ and 
$|\mathcal{G}|,|\mathcal{A}|< \mathfrak{d}$, then there is a function $f\in \baire$ so that for any $g\in \mathcal{G}$ and $a\in \mathcal{A}$, $|\{ n\in a : f(n) > g(n) \}|= \omega$. 
\end{lemma}

\begin{de}[Roitman \cite{roitman2011paracompactness}]\label{MH}
The \emph{Model Hypothesis}, abbreviated $\mathsf{MH}$, is the following statement: For some $\kappa$, $H(\om_1)$ is the increasing union of $H_\alpha$'s, for $\alpha < \kappa$, where each $H_\alpha$ is an elementary submodel of $(H(\om_1), \in )$ and each $H_\alpha \cap \baire$ is not dominating.
\end{de}

Here $H(\kappa)$ is the collection of all sets whose transitive closures have size less than $\kappa$. In particular, both $\baire, \propp (\om)$ are contained in $H(\om_1)$, and a space of countable weight can be coded as (hence is homeomorphic to) a subset of $H(\om_1)$.

\subsection{Box and Nabla Products}

We will follow Roitman's notation from  \cite{roitman2011paracompactness}. For $x\in \square_n X_n$,  we  write $\xbar$ for its mod-finite equivalence class, $[x]_\sim$, in $\nabla_n X_n$. 
If $x\in \square_n X_n$ or $\xbar \in \nabla_n X_n$, and $U = \langle U_n \rangle_{n\in \om}$ is a sequence of open sets, where $x(n)\in U_n \sub X_n$, define the basic neighborhood around $x$ 
by $N(x, U) := \square_n U_n \sub \square_n X_n$ and $N(\xbar, U) := \nabla_n U_n \sub \nabla_n X_n$. 
If the $X_n$'s are first countable, a basis of $x$ or $\xbar$ is coded by $\baire$ as follows: if $\{ U^k_n : k\in \om \}$ is a base at $x(n)$, we will write $N(x, f)$ for $\square_n U^{f(n)}_n$ and $N(\xbar, f)$ for $\nabla_n U^{f(n)}_n$, where $f\in \baire$. Following Roitman, we do not distinguish between elements of $\square_n X_n$ and $\nabla_n X_n$ ($x$ \textsl{versus} $\xbar$) if there is no chance of confusion.

A space $X$ is said to be $P_\kappa$ if the intersection of strictly fewer than $\kappa$-many open sets is open. 
We recall: every nabla product, $\nabla_n X_n$, is a $P_{\om_1}$-space; and if each $X_n$ is first countable, then $\nabla_n X_n$ is a $P_{\mathfrak{b}}$-space.


\subsection{Monotone Normality and Halvablility} 

Let $A$ be a subspace of a space $X$.
We say that $A$ is \emph{monotonically normal in $X$} if for every point $x$ of $A$ and set $U$ open in $X$ containing it, there is assigned an open (in $X$) set $G(x,U)$  such that $x\in G(x, U) \subseteq U$, and if $G(x, U) \cap G(y, V) \neq \emptyset$, then $x \in V$ or $y \in U$. 
Observe that for any base for $X$  we only need define $G(x,U)$ for basic $U$, and we may assume that $G(x,U)$ is in the base. This will be used frequently in the sequel. 
If $A$ is monotonically normal in some $X$ then clearly $A$ is monotonically normal.

A function $F$ on $A$ is a \emph{neighborhood assignment} (or, \emph{neighbornet}) for $A$ (in $X$) if $F(x)$ is a neighborhood of $x$, for every $x$ in $A$. (A neighbornet for the whole space is just called a `neighbornet'.) 
A neighbornet $T$ on $A$  is \emph{halvable in $X$} if there is a neighbornet $S$ for $A$ such that: if $S(x) \cap S(y) \ne \emptyset$ then $x \in T(y)$ or $y \in T(x)$. Note that we may assume that every $S(x)$ comes from any given base for $X$. The subspace $A$ of   $X$ is \emph{halvable in $X$} if every neighbornet of $A$ in $X$ is halvable.  A space is \emph{halvable} if it is halvable in itself (every neighbornet can be halved).

Observe that if $A$ is monotonically normal in $X$ then it is halvable in $X$, and so every monotonically normal space is halvable. 
The converse is false. For example, every countable (Tychonoff) space is halvable \cite{KPHart}, but there are countable spaces that are not monotonically normal (for example, all polynomials with rational coefficients with the topology of pointwise convergence, see \cite{pg}). However, it turns out that in certain cases nabla products are monotonically normal if they are halvable, indeed it suffices that just one specific neighbornet be halvable.

\begin{lemma}\label{LemmaMNequivalentToHalvable}
Let $X$ be a space with partial order $\preceq$ and neighborhood bases, $\mathcal{B}_x$, for each $x\in X$, such that: (a) $\downarrow x=\{ y : y \preceq x\}$ is a neighborhood of $x$ for all $x$, and (b) if $y \in B \subseteq \; \downarrow x$, where $B \in \mathcal{B}_x$, then the interval $[y,x]=\{ z : y \preceq z \preceq x\}$ is contained in $B$. Let $A$ be a subspace of $X$. Then, $A$ is monotonically normal in $X$ if and only if the neighbornet $T_A(x) = \; \downarrow x$, for $x$ in $A$, is halvable in $X$.
\end{lemma}
\begin{proof}
We only need to show that if $S$ halves $T_A(x) =  \; \downarrow x$, then $A$ is monotonically normal in $X$. For any element $x$ of $A$ in some $B \in \mathcal{B}_x$, where $B \subseteq \; \downarrow x$, define $G(x,B) = S(x) \cap B$. We prove that this is a monotone normality operator for $A$ in $X$. Suppose $z \in G(x,B) \cap G(x',B')$. Then $S(x)$ meets $S(x')$, and suppose without loss of generality that $x' \in \; \downarrow x$, that is, $x' \preceq x$. As $z \in B \sub \; \downarrow x$, we have $z \preceq x$. Hence, $[z,x]$ is contained in $B$. But as $z \in B' \sub \; \downarrow x'$, we have $z \preceq x'$. Hence $x'$ is in $[z,x]$, and so in $B$.  
\end{proof}

A space $X$ is \emph{$\kappa$-metrizable} (for a cardinal $\kappa \ge \omega_1$) if it has an open base $\mathcal{B} = \{ U_{x,\alpha} : \alpha < \kappa, \; x\in X \}$ so that $\{ U_{x,\alpha} : \alpha < \kappa \}$ is a neighborhood base at $x$, 
and given two points $x, y$ and two ordinals $\alpha \leq \beta < \kappa $ then (i) if $y\in U_{x, \alpha}$ then $U_{y, \beta} \sub U_{x, \alpha}$; 
and (ii) if $y \notin U_{x, \alpha}$ then $U_{y, \beta} \cap U_{x, \alpha} = \emptyset$. Every $\kappa$-metrizable space is paracompact and monotonically normal.


\subsection{Not Hereditarily Normal}
We observe that certain spaces are not hereditarily normal. These will be used as test spaces to show certain nabla products are not hereditarily normal. The results are probably folklore, so we sketch just enough of their proofs for the full argument to be reconstructed by the reader.

If $\lambda, \kappa$ are cardinals, denote by $D(\kappa)$ the discrete space of size $\kappa$ and let $L_\lambda(\kappa)$ be the space with underlying set $D(\kappa) \cup \{ \kappa \}$, and topology where points of $D(\kappa)$ are isolated and neighborhoods around $\kappa$ have the form $\{ \kappa \} \cup (D(\kappa) \setminus C)$ for $C\sub D(\kappa)$ of size less $< \lambda$. Write $A(\kappa)$ for $L_{\om}(\kappa)$, the one-point compactification of $D(\kappa)$, and $L(\kappa) = L_{\om_1}(\kappa)$ the one-point Lindelofication of $D(\kappa)$.

\begin{lemma}\label{L(omega2)notHereditarilyNorlmal}
$L(\om_2)^2$ is not hereditarily normal.
\end{lemma}
\begin{proof}
More precisely, $Y = L(\om_2)\times L(\om_2) \setminus \{ (\om_2, \om_2) \}$ is not normal, because the sets $H= (L(\om_2) \setminus \{ \om_2\}) \times \{ \om_2 \}$ and $K= \{ \om_2 \} \times (L(\om_2) \setminus \{ \om_2\})$ are disjoint and closed in $Y$, and can not be separated by disjoint open sets.

Indeed suppose $U$ and $V$ be any open neighborhoods around $H$ and $K$, respectively. For every $(\alpha, \om_2) \in H$ choose $A_\alpha \in [D(\om_2)]^\om$ such that $\{ \alpha \} \times (L(\om_2) \setminus A_\alpha) \sub U$, and similarly, for $(\om_2, \beta) \in K$ choose sets $B_\beta$ such that $(L(\om_2) \setminus B_\beta) \times \{ \beta \} \sub V$. 
Then there is $\delta \geq \om_1$ such that for every $\alpha \le \delta$, $A_\alpha \sub \delta$.  (To see this, let $M \prec H(\om_3)$ be an elementary submodel of size $\om_1$ such that $\{ A_\alpha : \alpha < \om_2 \} \in M$, set $\delta = M \cap \om_2 \in \om_2$, and now, for $\alpha \in M$, $M$ thinks `$A_\alpha$ is contained in $M$', and so does $H(\om_3)$.) Now a counting argument easily shows $U$ and $V$ meet.
\end{proof}

\begin{lemma}\label{LemmaSxSNotNormal}
Let $S$ be a stationary subset of a regular uncountable cardinal $\kappa$. Then, $S \times (S \cup \{\kappa\})$ (as a subspace of $(\kappa+1)^2$) is not normal.
\end{lemma}
\begin{proof}
Consider the diagonal $H = \{ (\alpha, \alpha) : \alpha \in S \}$ and the top-edge $K= \{ (\alpha , \kappa ) : \alpha \in S \}$. Note that $H$ and $K$ are closed disjoint sets. Now, if $U, V$ are neighborhoods of $H,K$, respectively, 
a standard Pressing Down Lemma argument shows that $U$ and $V$ must meet.
\end{proof}


\section{Embeddings into Nabla Products}\label{Stationary}

The following simple embedding result will be used frequently in the sequel. For a space $X$ let $X_\delta$ be the $G_\delta$-modification of $X$ (the space with underlying set $X$ and topology generated by all $G_\delta$ subsets of the space $X$).
\begin{lemma}[Williams \cite{williams1984box}, Lemma~4.4]\label{LemGdeltaEmbed}
Let $X$ be a space. Then $X_\delta$ embeds as a closed subspace in $\nabla X^\omega$ via the map $x \mapsto c_x$, where $c_x$ is constantly equal to $x$.
\end{lemma}

However the main technical result of this section is about non-embedding.
\begin{prop}\label{embedding of S}
Let $\{ X_n: n\in \om\} $ be a family of metrizable spaces and $S$ a stationary subset of a regular uncountable cardinal $\kappa$. Then $S$ does not embed into $\nabla_{n} X_n$.
\end{prop}

Since Balogh and Rudin \cite{balogh1992monotone} showed that  a monotonically normal space is paracompact if and only if it does not contain closed copies of stationary subsets of regular uncountable cardinals, we deduce:
\begin{thr}\label{ThrNablaMNimpliesHP}
Let $\{ X_n: n\in \om\} $ be a family of metrizable spaces. If a subspace $A$ of $\nabla_n X_n$ is monotonically normal then $A$ is hereditarily paracompact.
\end{thr}

\begin{proof}[Proof of \cref{embedding of S}]  Suppose, for a contradiction,  $\varphi: S \to \nabla_n X_n$ is an embedding. We split the proof in two cases depending on the size of $\kappa$. 
If $\kappa \leq \mathfrak{d}$, then any $\alpha$ in 
$\mathop{Lim}(S)$ has $\mathop{cf}(\alpha) \neq \mathfrak{d}$, so $S$, and $A=\varphi(S)$, have limit points but no points of character $\mathfrak{d}$, contradicting \cref{tightness}.
If $\kappa > \mathfrak{d}$, by \cref{eventually} the map $\varphi$ is eventually constant, thus it cannot be an embedding.
\end{proof}

Recall that the \emph{character} of a topological space $X$ at a point $x$ is the cardinality $\chi (x, X)$ of the smallest local base for $x$. The \emph{tightness} at a point $x$ in $X$, denoted $t(x, X)$, is the smallest cardinal  $\kappa$ such that whenever $x\in \overline{Y}$ for some $Y\sub X$, there exists a subset $Z\sub Y$, with $x \in \overline{Z}$ and $|Z| \leq \kappa$.

\begin{lemma}\label{tightness}
Let $\{ X_n : n\in \om\}$ be a family of first countable spaces. For any point $x\in \overline{A} \setminus A$, where $A \subseteq \nabla_n X_n$, we have that $t(x, A \cup \{ x \})= \mathfrak{d} = \chi (x, A \cup \{ x\})$.
\end{lemma}

\begin{proof}
It is easy to check that $t(x,A \cup \{x\}) \leq \chi(x,A \cup \{x\}) \leq \chi (x, \nabla_n X_n) = \mathfrak{d}$ (the last equality holds because a local basis of $x$ can be represented by a dominating family of $\baire$). We only have to prove that $t(x,A \cup \{x\}) \ge  \mathfrak{d}$. 

If $x$ is a limit point of $A$, then for infinitely many $n\in \omega$, $x(n)$ is non-isolated in $X_n$. Thus, without loss of generality we may assume that for all $n\in \om$, $x(n)$ is non-isolated in $X_n$. Let $\{ B_m (n) : m \in \omega \}$ be a decreasing countable local basis for $x(n)$, for every $n\in \om$.

Suppose for a contradiction that there is $Z \subseteq A$ with $|Z| < \mathfrak{d}$ and $x\in cl_{A\cup \{x\}} (Z)$. For every $z\in Z$, there is an infinite set $a_z \subseteq \omega$ such that for $n\in a_z$, $z(n) \neq x(n)$. Also, for every $z \in Z$ there is a function $f_z \in \omega^\omega$ such that for $ n\in a_z$, $z(n) \notin B_{f_z(n)} (n)$. Thus, $z\notin N(x, f_z)$. Let $\mathcal{G} = \{ f_z : z\in Z\}$ and $\mathcal{A} = \{ a_z : z\in Z \}$. By \cref{LemmaDiagonilizingFunctions}, there is $f\in \omega^\omega$ diagonalizing the families $\mathcal{G}$ and $\mathcal{A}$. Then, for any $z\in Z$, $z\notin N(x, f)$, contradicting  that $x \in cl_{Z \cup \{ x \}}(Z)$.
\end{proof}

\begin{lemma}\label{eventually}
Let $\{ ( X_n, d_n) : n\in \om \}$ be a family of metric spaces and $S$ a stationary subset of a regular uncountable cardinal $\kappa >\mathfrak{d}$. Then, every continuous function $\varphi: S \to \nabla_n X_n$ is eventually constant.
\end{lemma}
\begin{proof} For each $n\in \om$, write $B_n(a,\varepsilon)$ for the $\varepsilon$-ball around $a$ in $X_n$ with respect to the metric $d_n$. 
Let $\{ f_\mu : \mu < \mathfrak{d} \} \subseteq \omega^\omega$ be a dominating family. For every $x\in \nabla_n X_n$ and $\mu < \mathfrak{d}$, define $N(x, f_\mu) = \nabla_{n\in \omega} B_n (x(n) , \rfrac{1}{f_\mu (n)})$; $\{ N(x, f_\mu) : \mu < \mathfrak{d}\}$ is a local basis at $x$.  

Fix, for the moment, $\mu < \mathfrak{d}$. 
For every $\alpha \in Lim(S)$, pick $g_\mu (\alpha) < \alpha$, $g_\mu (\alpha) \in S$, such that $\varphi[(g_\mu (\alpha), \alpha]] \subseteq N(\varphi(\alpha), 2 f_\mu)$. Then, $g_\mu$ is a regressive function and by the Pressing Down Lemma, there is $\alpha_\mu$ in $S$ and a stationary set $S_\mu \subseteq S$ such that for any $\beta \in S_\mu, g_\mu(\beta) = \alpha_\mu$.

We claim that for all $\delta, \gamma > \alpha_{\mu}$, where $\delta$ and $\gamma$ are in $S$, $\varphi(\gamma)$ is in $N(\varphi(\delta), f_{\mu})$. 
To see this, take any $\delta$ and $\gamma$ strictly larger than $\alpha_\mu$ in $S$. 
Since $S_\mu$ is stationary, there is a $\beta$ in $\mathop{Lim}(S_\mu)$ with $\beta > \max\{ \gamma, \delta \}$. Then, $\{\varphi(\gamma), \varphi(\delta) \} \subseteq \varphi[(\alpha_\mu, \beta]] = \varphi[(g_{\mu}(\beta), \beta]] \subseteq N(\varphi(\beta), 2 f_{\mu})$. 
By definition of $N(\varphi(\beta), 2 f_{\mu})$, for all  $n\in  \omega$  we have that $\varphi(\gamma)(n)$ and $\varphi(\delta)(n)$ are in $B_n(\varphi(\beta)(n), \rfrac{1}{2 f_{\mu} (n)})$. 
Then by symmetry and the triangle inequality, for all $n\in \om$, we get 
$d_n(\varphi(\gamma)(n), \varphi(\delta)(n)) < \rfrac{1}{f_{\mu} (n)}$. This implies $\varphi(\gamma) \in N(\varphi(\delta), f_{\mu})$, as claimed.

Now, as we let $\mu$ run over all values below $\mathfrak{d}$, since $\kappa > \mathfrak{d}$, there is a least upper bound $\alpha_\infty$ of $\{ \alpha_\mu : \mu < \mathfrak{d} \}$ in $S$. Notice that by the claim above,  for any $\mu < \mathfrak{d}$ and $\gamma, \delta \in S \setminus \alpha_\infty$, we have $\varphi(\gamma) \in N(\varphi(\delta), f_\mu)$, and so $\varphi(\gamma) = \varphi(\delta)$. Hence $\varphi$ is constant from $\alpha_\infty$ on, as desired.
\end{proof}

\section{\texorpdfstring{$\Delta$}{Delta} and \texorpdfstring{$\nabla (\om+1)^\om$}{Nabla}}
\label{S:Delta_and_Nabla}


\subsection{Roitman's Principle \texorpdfstring{$\Delta$}{Delta}}\label{Delta}

In order to state our parametrized versions of Roitman's $\Delta$ principle we introduce some specific definitions and notation, naturally extending those of Roitman, for partial functions.

For any function $x: N \to \omega$, where $N \subseteq \omega$, consider $x$ to be a \emph{partial function} from $\omega$ to $\omega$ and write $\mathop{dom} x$ for $N$, the domain of $x$. We identify a partial function with its graph, which is a subset of $\omega \times \omega$. Then two partial functions, $x$ and $y$, are \emph{almost equal},  $x =^* y$, if $x\setminus y$ and $y \setminus x$ are both finite. 
Let $\omega^{\subseteq \omega}$ be the set of all partial functions, including the empty function. 
Denote by $\omega^{\subset \omega}$ the set of all partial functions whose domain  is infinite and co-infinite. 
For any subset $A$ of $\om^{\subseteq \om}$ let $A^* = \{ y \in \om^{\subseteq \om} : y =^* x$ for some $x\in A\}$. 
For $k\leq \om$, let $c_k$ be the constant $k$-valued function in $(\om+1)^\om$. 
Let $x$ be a partial function and $h\in \baire$, then we say that $x >^* h$  if for all but finitely many $n\in dom(x)$, $x(n) > h(n)$.

\begin{de}
Two partial functions, $x$ and $y$, \emph{switch} if $|x\setminus y|= |y\setminus x| = \omega$ and $x(n)=y(n)$ for all but finitely many $n$ in $\mathop{dom} x \cap \mathop{dom} y$.
\end{de}

\begin{de}\label{delta}
Let $A$ be any subset of $\omega^{\subseteq \omega}$. Then \emph{$\Delta(A)$} is the statement:
there exists  $F : A \to \omega^\omega$ such that if $x,y \in A$ switch, then $x\setminus y \ngtr^* F(y)$ or $y\setminus x \ngtr^* F(x)$.
\end{de}

\begin{lemma}\label{L:EqvtDeltas}
Let $A$ and $B$ be subsets of $\omega^{\subseteq \omega}$. Then: 

(1) if $A \subseteq B \subseteq \omega^{\subseteq \omega}$ then $\Delta (B) \implies \Delta (A)$, and

(2) $\Delta(A) \iff \Delta (B)$  when $A \cap \omega^{\subset \omega} \subseteq B \subseteq (A \cup \om^\om \cup \{\emptyset\})^*$.
\end{lemma}
\begin{proof}
Claim (1) is clear -- simply restrict an $F$ witnessing $\Delta (B)$ to $A$ to get a witness of $\Delta (A)$.

For claim (2) it suffices to show: $\Delta (A \cap \omega^{\subset \omega}) \implies \Delta ((A \cup \om^\om \cup \{\emptyset\})^*)$. Fix $F$ as in $\Delta (A \cap \om^{\subset \om})$. 
Note that if $x$ is in $\om^\om \cup \{\emptyset\}$ then it does not switch with \emph{any} $y$ in $\om^{\subseteq \om}$. 
So we can extend $F$ over $\om^\om \cup \{\emptyset\}$ completely arbitrarily, and it witnesses $\Delta (A \cup \om^\om \cup \{\emptyset\})$.   
Note that if $x =^* x'$ and $y=^* y'$ then $x, y$ switch if and only if $x',y'$ switch, and similarly for the conclusion of $\Delta$. So we can extend $F$ over $A^*$ in the natural way (if $x \in A$, $x=^*x'$ and $x' \notin A$ then set $F(x')=F(x)$) to get a witness of $\Delta ((A \cup \om^\om \cup \{\emptyset\})^*)$, as required.
\end{proof}

Abbreviate $\Delta(\omega^{\subset \omega})$ to $\Delta$, this is Roitman's combinatorial principle in \cite{roitman2011paracompactness} and \cite{roitman2015paracompactness}. 
It is known to be consistently true (under $\mathfrak{b}=\mathfrak{d}$, $\mathfrak{d}=\mathfrak{c}$, $\mathsf{MH}$, and in any forcing extension obtained by adding cofinally many Cohen reals) but it is unknown if it can be consistently false, or it is true in $\mathsf{ZFC}$.

In \cite{roitman2011paracompactness} Roitman showed that  $\Delta$ implies the subspace $\nabla^* = \nablaomega \setminus (\nabla \om^\om \cup \{\topl\})$ of $\nablaomega$ is paracompact. Then, she claimed, without proof, that `$\nabla^*$ is paracompact if and only if $\nablaomega$ is paracompact'. Here lies the gap. Adding  isolated points (like those of $\nabla \om^\om$) to even the best behaved of spaces (discrete, for example) frequently destroys normality and paracompactness (indeed, many classical counter-examples related to normality have this form).

\subsection{\texorpdfstring{$\Delta (A)$}{Delta(A)}-Principles That Hold in \texorpdfstring{$\mathsf{ZFC}$}{ZFC}}

While we only know of consistency proofs for $\Delta=\Delta(\omega^{\subset \omega})$,  there are, however, interesting $A$ such that $\Delta(A)$  is true in $\mathsf{ZFC}$. We present here an example.

Let $\mathrm{INC}$ be the set of all increasing partial functions (so $x \in \om^{\subset \om}$ is in $\mathrm{INC}$ if whenever $m \le n$ are in $\mathop{dom}{x}$ we have $x(m) \le x(n)$). 
Let $\mathrm{FI}$ be all partial functions which are the finite disjoint union of increasing functions 
(so $x$ is in $\mathrm{FI}$ if $\mathop{dom}{x}$ can be partitioned into $S_1, \ldots, S_k$ such that each $x \restriction S_i$ is in $\mathrm{INC}$). We will show: 
 the combinatorial principle $\Delta(\mathrm{FI})$ is true in $\mathsf{ZFC}$.

First we need to show that elements of $\mathrm{FI}$ have a nice representation.
Let $x$ be in $\om^{\subseteq \om}$. Define $\perp x = \{ m \in \mathop{dom} x : \forall n > m, \ x(m) \le x(n)\}$. Set $x^\perp = x \restriction (\perp x)$.
Observe that $x^\perp$ is increasing.
Set $x_0=x^\perp$; and inductively, $x_n = (x \setminus \bigcup_{i<n} x_i)^\perp$. 

\begin{prop}\label{P:FI_rep}
Let $x$ be in $\mathrm{FI}$, say $x = x^0 \cup \cdots \cup x^\ell$ where each $x^i$ is increasing. Then  $x=x_0 \cup \cdots \cup x_\ell$. 
\end{prop}
\begin{proof}
Evidently each $x_k$ is a subset of $x$, so we need to show $x \subseteq x_0 \cup \cdots \cup x_\ell$. We do so by breaking $x$ into finite pieces, ${}^i x$, each of which is contained in $\bigcup_{k \le \ell} x_k$. 

For any $y \in \om^{\subseteq \om}$ set $\mathop{Dec}(y) \subseteq \mathop{dom} y$ such that $\mathop{min}(\mathop{dom} y) \in Dec(y)$ and $n\in Dec(y)$ with $n > \mathop{min}(\mathop{dom} y)$ if and only if for all $m< n$ and $m\in \mathop{dom} y$, $y(m) > y(n)$. Note that $\mathop{Dec}(y)$ is finite and non-empty, and $y \restriction Dec(y)$ is \emph{strictly} decreasing.
Set ${}^0x = x \restriction \mathop{Dec}(x)$; and inductively, ${}^n x = x \restriction \mathop{Dec}(x \setminus \bigcup_{i  <n } \ {}^i x)$. Note that $x$ is the disjoint union of the ${}^i x$'s.  
Since $x$ is the union of $\ell+1$ increasing functions, the $x^i$,  there can not be a strictly decreasing portion of $x$ of size greater that $\ell+1$.
Hence, for every $i\in \omega$, $0\leq |{}^i x| -1 := \ell_i \leq \ell$.  
Enumerate in increasing order, 
$\mathop{dom} \ {}^i x = \langle n^i_j : j \leq \ell_i \rangle$ (hence, ${}^i x(n^i_j) >  {}^i x(n^i_{j+1})$).
The next claim shows that each ${}^i x$ is contained in $x_0 \cup \cdots \cup x_\ell$, which completes the proof.

\noindent {\bf Claim:} for every $j \leq \ell_i$ and $i\in \omega$, $(n^i_{\ell_i -j} , \ ^i x (n^i_{\ell_i -j}))$ is in  $\bigcup_{k\leq j} x_k$. 

We will proceed by induction on $j$.
First $j=0$. Take any $i\in \omega$. 
The last point $(n^i_{\ell_i} , ^i x (n^i_{\ell_i}))$ has the property that $\forall m > n^i_{\ell_i}$, $^i x (n^i_{\ell_i}) = x (n^i_{\ell_i}) \leq x(m)$. 
Otherwise, if there is  $m > n^i_{\ell_i}$ such that $^i x (n^i_{\ell_i}) > x(m)$, pick $m_i$ the minimum $m > n^i_{\ell_i}$ with that property, and we get $(m_i, x(m_i)) \in \ ^i x$, which contradicts the construction of $^i x$. Hence, $(n^i_{\ell_i -j} ,\ ^i x (n^i_{\ell_i -j})) \in x^j$.

Now suppose that  $\forall m< j$, $\forall i\in \omega$ we have $(n^i_{\ell_i -m} ,\ ^i x (n^i_{\ell_i -m})) \in \bigcup_{k < j}  x_k$. We need to prove that $(n^i_{\ell_i -j} ,\ ^i x (n^i_{\ell_i -j})) \in x_{j}$. The point $(n^i_{\ell_i -j} ,\ ^i x (n^i_{\ell_i -j}))$ has the property: $\forall m > n^i_{\ell_i -j}$, $^i x (n^i_{\ell_i -j}) \leq (x \setminus \bigcup_{k \leq j} x_k )(m)$.

Suppose for a contradiction that $m>n^i_{\ell_i -j}$ is the least such that $^i x (n^i_{\ell_i -j})$ is strictly greater than $(x \setminus \bigcup_{k \leq j} x_k )(m)$. There is an $i'> i$ such that $m \in \mathop{dom}  \ ^{i'}x$. 
By inductive hypothesis, we have removed the last $j$ points (from $0$ to $j-1$) of $^i x$, for all $i\in \omega$. Hence, $m$ has the form $n^{i'}_{\ell_{i'} -j}$. 
Consider the decomposition $x = x^0 \cup \ldots \cup x^\ell$, where the $x^k$ are in $\mathrm{INC}$. Observe that for every $i\in \omega$ and $k\leq \ell$, $|x^k \cap \ ^i x| \leq 1$. Also, notice that for every $s \leq \ell_i -j$, if $n^i_{\ell_i -(j+s)} \in \mathop{dom} x^k$, then $n^{i'}_{\ell_{i'} -j} \notin \mathop{dom} x^k$, because $x^k$ is increasing and $x(n^i_{\ell_i -j}) > x(m)$ (our assumption). Thus, there is a one-to-one correspondence between the last $j$ points of $^i x$ and the last $j+1$ points of $^{i'} x$, the desired contradiction.
\end{proof}

We will say that two elements $z,w$ of $\om^{\subseteq \om}$ are \emph{compatible} if for all but finitely many $n$ in $\mathop{dom} z \cap \mathop{dom} w$ we have $z(n)=w(n)$. 
Now let $x$ be an element of $\om^{\subseteq \om}$ with infinite domain.
Note that $x^\perp$ also has infinite domain.
For each $n \notin \mathop{dom} x$ let $n^{+x}$ be the minimal element of $\mathop{dom} x$ larger than $n$. 
Let $\mathrm{INC}^+$ be the set of members of $\mathrm{INC}$ with infinite domain.
For $x$  in $\mathrm{INC}^+$, define $F(x)$ by $F(x)(n)$ is $x(n)$ when $n \in \mathop{dom} x$ and is $x(n^{+x})+1$ otherwise.

\begin{lemma}\label{L:IndStep}
Let $z,w$ be compatible elements of $\om^{\subseteq \om}$ with infinite domain. If $z \setminus w >^* F(w^\perp)$ and $w \setminus z >^* F(z^\perp)$ then $w^\perp =^* z^\perp$
\end{lemma}
\begin{proof}
We will show that $z^\perp =^* u^\perp$ where $u= (w \setminus z) \cup z$. Then symmetrically, $w^\perp =^* v^\perp$ where $v= (z \setminus w) \cup w$, hence by compatibility of $z,w$ we have $u =^* v$, and so $z^\perp =^* w^\perp$, as claimed.
For $z^\perp =^* u^\perp$ it suffices that $\perp z =^* \perp u$, because $u$ equals $z$ on $\perp z \subseteq \mathop{dom} z$, and so $u^\perp = u \restriction (\perp u) =^* z \restriction (\perp z) = z^\perp$. 
Fix $N$ such that for all $n \ge N$ we have $(w \setminus z) (n) > F(z^\perp)(n)$. 

First a general observation: if $k \in \mathop{dom} (w \setminus z)$ and $k \ge N$, then $k \le k^{+z^\perp}$ and 
\[ u(k) = (w \setminus z) (k) > F(z^\perp)(k) \geq z^\perp (k^{+z^\perp}) = u(k^{+z^\perp}).\]
Now take any $n$ in $\perp u$ with $n \ge N$. By the observation $n$ must be in $\mathop{dom} z$ (not $\mathop{dom} (w \setminus z)$) and clearly (as $z \subseteq u$) $n$ is in $\perp z$.

For the other inclusion, take any $n\ge N$ in $\perp z$. Take any $m \in \mathop{dom} u$, $m \ge n$.
If $m \in \mathop{dom} z$ then $u(n) \le u(m)$ because $n \in \perp z$.
Otherwise $m \in \mathop{dom} (w\setminus z)$, then by the observation, $u(m) > z^\perp (m^{+z^\perp}) \ge z(n)$ (as $n \in \perp z$), and $z(n)=u(n)$.
Either way, $u(n) \le u(m)$. Thus $n \in \perp u$, as required.
\end{proof}

We now show $\Delta (\mathrm{FI})$ is true. From \cref{L:EqvtDeltas} we deduce that in fact $\Delta ((\mathrm{FI} \cup \baire \cup \{\emptyset\})^*)$ is true. Note that a partial function is in $\mathrm{FI}^*$ if it is eventually in $\mathrm{FI}$, or equivalently, if it is the finite union of eventually increasing partial functions.

\begin{prop}\label{P:DeltaFI} 
In $\mathsf{ZFC}$ we have $\Delta (\mathrm{FI})$.
\end{prop}

\begin{proof}
Take any $x'$  in $\mathrm{FI}$. Then by \cref{P:FI_rep}, $x'$ is the disjoint union $x'_0 \cup x'_1 \cup \cdots \cup x'_{\ell'}$.
Note that $x' =^* x$ where  $x$ is either the empty set, or the disjoint union $x_0 \cup x_1 \cup \cdots \cup x_{\ell}$, and each of $x_0, \ldots, x_\ell$ is in $\mathrm{INC}^+$. 
Let $\mathrm{FI}^+$ be all $x$ in $\mathrm{FI}$ such that all of $x_0, \ldots, x_\ell$ are in $\mathrm{INC}^+$.
Then we have just said that $\mathrm{FI} \subseteq (\{\emptyset\} \cup \mathrm{FI}^+)^*$, hence, by \cref{L:EqvtDeltas}, to show $\Delta (\mathrm{FI})$ it suffices to prove $\Delta (\mathrm{FI}^+)$.
For $x$ in $\mathrm{FI}^+$ define $F(x)$ to be the maximum of $F(x_0), \ldots, F(x_\ell)$. 

Take any $x,y$ in $\mathrm{FI}^+$. They have representation $x=x_0 \cup \cdots \cup x_\ell$ and $y=y_0 \cup \cdots \cup y_m$.
Assume, without loss of generality, that $\ell \le m$. 
Suppose $x$ and $y$ are compatible, $x \setminus y >^* F(y)$ and $y \setminus x >^* F(x)$. To establish $\Delta (\mathrm{FI}^+)$ we show $x \setminus y$ is finite (so $x,y$ do not switch), because $x_n =^* y_n$ for all $n \le \ell$, which we verify by induction on $n$.

Inductively, suppose $x_i =^* y_i$ for all $i<n$. 
Let $z=x \setminus \bigcup_{i<n} x_i$ and $w=y \setminus \bigcup_{i<n} y_i$. 
Then $z =^* x \setminus \bigcup_{i<n} y_i$ and $w =^* y \setminus \bigcup_{i<n} x_i$.
Note that $z^\perp = x_n$ and $w^\perp = y_n$. 
Since $x$ and $y$ are compatible so are $z$ and $w$. 
Note that $z \setminus w =^* (x \setminus \bigcup_{i<n} y_i) \setminus (y \setminus \bigcup_{i<n} y_i) = x \setminus y$.
Hence,  $z \setminus w =^* x \setminus y  >^* F(y) \ge F(y_n)=F(w^\perp)$.
Symmetrically, $w \setminus z >^* F(z^\perp)$.
Thus $z$ and $w$ satisfy the hypotheses of \cref{L:IndStep}, and we see $x_n = z^\perp =^* w^\perp = y_n$, as claimed. 
\end{proof}

\subsection{Halvability and Monotone Normality of \texorpdfstring{$\nabla$}{Nabla}}\label{MNNabla}

We now connect the combinatorial principle $\Delta (A)$ with the topology of $\nablaomega$. 
There is a natural bijection between $(\om+1)^\om$ and $\omega^{\subseteq \omega}$. Indeed given
 $x\in \omega^{\subset \omega}$, we can extend it to $x'$ in $(\omega+1)^\om$  by giving $x'$ value $\omega$ outside the domain of $x$.
 Conversely, given $x$ in $(\om+1)^\om$ we get an element of $\omega^{\subseteq \omega}$ by restricting it to $N=\{n \in \omega : x(n) \in \omega\}$. \emph{Throughout this section we identify $\boxomega$ with $\omega^{\subseteq \omega}$.} 
Recall if $x\in \boxomega$, then we  write $\xbar$ for its equivalence class, $[x]_\sim$, in $\nabla (\om+1)^\om$. We extend this
according to our identification, and given $x$ in $\omega^{\subseteq \omega}$, write $\xbar$ for $[x']_\sim$. 
(Note that $\overline{\emptyset} = \topl$.) 

The set $\nablaomega$ has a natural partial order: for $x$ and $y$ in $\omega^{\subseteq \omega}$ write $\ybar \preceq \xbar$ if and only if for all but finitely many $n$ in $\mathop{dom} x$ we have $y(n)=x(n)$. Note that $\topl$ is the $\preceq$-largest element of $\nablaomega$.
We say that $\xbar$ and $\ybar$ are \emph{compatible} if they have a common $\preceq$-lower bound. 
Proof of the next lemma just requires chasing definitions, and is left to the reader.

\begin{lemma}\label{L:Nabla_technical0} 

Let $x,y$ be in $\omega^{\subseteq \omega}$.

(i) $\xbar$ and $\ybar$ are compatible if and only if for all but finitely many $n$ in $\mathop{dom} x \cap \mathop{dom} y$ we have $x(n)=y(n)$. 

(ii) If $\xbar$ and $\ybar$ are compatible then they have a greatest lower bound, $\zbar = \xbar \wedge \ybar$, where $z = \{(n,k) : x(n)=k=y(n)\} \cup (x \setminus y) \cup (y \setminus x)$.

(iii) $\ybar \npreceq \xbar$ if and only if $\xbar$ and $\ybar$ are not compatible or $\xbar, \ybar$ are compatible but $x \setminus y$ infinite.
\end{lemma}

For any $x$ in $\omega^{\subseteq \omega}$ basic neighborhoods around $\xbar$ are of the form 
$N(\overline{x} ,h) = \{ \overline{y} \in \nablaomega : \ybar \preceq \xbar$  and $y\setminus x >^* h \}$, where 
$h\in \baire$.

\begin{lemma}\label{L:Nabla_technical} 

Take any $x,y$ in $\omega^{\subseteq \omega}$ and $f_x,f_y$ in $\baire$.

(i) (a) $\xbar \in N(\xbar,c_0) \subseteq \downarrow \xbar$ and (b) if $\ybar \in N(\xbar,f_x)$ and $\ybar \preceq \zbar \preceq \xbar$ then $\zbar \in N(\xbar,f_x)$.

(ii) $N(\xbar,f_x) \cap N(\ybar,f_y) \ne \emptyset$ if and only if $\xbar \wedge \ybar \in N(\xbar,f_x) \cap N(\ybar,f_y)$, if and only if $\xbar, \ybar$ are compatible, and $y \setminus x >^* f_x$ and $x \setminus y >^* f_y$.
\end{lemma}
\begin{proof}
Claim (i) (a) is evident. Towards (i) (b), suppose $\ybar \preceq \zbar$. Then for all but finitely many $n$ in $\mathop{dom} (z \setminus x) \subseteq \mathop{dom} z$ we have $(z \setminus x)(n)=z(n) = y(n)=(y \setminus x)(n)$. Hence if $y \setminus x >^*f_x$ then also $z \setminus x >^* f_x$. And (i) (b) follows.

For the first equivalence of (ii), note that if $\zbar$ is in $N(\xbar,f_x) \cap N(\ybar,f_y)$ then $\zbar$ is $\preceq$-below both $\xbar$ and $\ybar$. Now apply (i) (b). The second equivalence follows from the definitions.
\end{proof}

If $A$ is any subset of $\omega^{\subseteq \omega}$, write  
$\nabla(A)$ for the subspace $\{ \xbar : x\in A \}$ of $\nabla (\om+1)^\om$, set  $\nabla^*(A) = \nabla (A \cap \omega^{\subset \omega})$ and set  $\nabla^+(A) = \nabla ((A \cup \om^\om \cup \{\emptyset\})^*)$. 
Then $\nabla^*(\omega^{\subseteq \omega})=\nabla(\omega^{\subset \omega})$ is  $\nabla^*$ from above, and abbreviate $\nabla(\omega^{\subseteq \omega})= \nabla^+(\omega^{\subset \omega})= \nabla (\om+1)^\om$ to $\nabla$.

\begin{thr}\label{MN_technical} 
Let $A$ be a subset of $\omega^{\subseteq \omega}$. Then the following are equivalent:
(1) $\Delta(A)$ holds, (2) $\nabla^*(A)$ is halvable in $\nabla$,  and (3) $\nabla^+(A)$ is monotonically normal in $\nabla$.
\end{thr}
\begin{proof} 
From \cref{L:Nabla_technical}(i) we see that $\nabla$ with $\preceq$ and the standard basic neighborhoods satisfies conditions (a) and (b) of \cref{LemmaMNequivalentToHalvable}, and any subspace of $\nabla$ is monotonically normal in $\nabla$ if and only a specific neighbornet is halvable. 
Combining this with \cref{L:EqvtDeltas} we see that to prove the equivalence of (1) through (3) it is sufficient to show: $\Delta (A)$ holds if and only if the neighbornet $T(\xbar)=\downarrow \xbar$ for $x$ in $A$ is halvable in $\nabla$.

Suppose $F$ is a function from $A$ into $\baire$. Define the neighbornet $S$ of $\nabla (A)$ in $\nabla$ by $S(\xbar) = N(\xbar,f_x)$ where $f_x=F(x)$.
On the other hand, suppose $S$ is a  neighbornet of $\nabla (A)$ in $\nabla$. We may assume each $S(\xbar)$ is basic, say $S(\xbar) = N(\xbar,f_x)$. Define $F : A \to \baire$ by $F(x) = f_x$. We show $F$ witnesses $\Delta (A)$ if and only $S$ halves in $\nabla$ the neighbornet $T(\xbar)=\downarrow \xbar$ for $\xbar$ in $A$.

First let us note, `$x$ and $y$ in $A$ switch', reinterpreted in terms of $\xbar$ and $\ybar$ via \cref{L:Nabla_technical0}(iii), is equivalent to, `$\xbar, \ybar$ are compatible, but $\ybar \npreceq \xbar$ and $\xbar \npreceq  \ybar$'. 
Next, taking the contrapositive, `$S(\xbar)=N(\xbar,F(x))$ halves $T(\xbar)=\downarrow \xbar$ in $\nabla$' is equivalent to, `$\ybar \npreceq \xbar$ and $\xbar \npreceq  \ybar$ implies $N(\xbar,F(x)) \cap N(\ybar,F(y)) = \emptyset$'. 

Now applying \cref{L:Nabla_technical0}(iii) and \cref{L:Nabla_technical}(ii), we see that `$S(\xbar)=N(\xbar,F(x))$ halves $T(\xbar)=\downarrow \xbar$ in $\nabla$' is equivalent to, `($\xbar, \ybar$ not compatible) or ($\xbar, \ybar$  compatible and $x \setminus y$ infinite and $y \setminus x$ infinite) implies ($\xbar, \ybar$ not compatible) or ($x \setminus y \ngtr^* F(y)$ or $y \setminus x \ngtr^* F(x)$)', which is equivalent to, `if ($\xbar, \ybar$  compatible and $x \setminus y$ infinite and $y \setminus x$ infinite) then ($x \setminus y \ngtr^* F(y)$ or $y \setminus x \ngtr^* F(x)$)', which (by the reinterpretation of switching above) is equivalent to `$F$ witnesses $\Delta (A)$'.
\end{proof}

\begin{thr}\label{MN}
Let $A$ be a subset of $\om^{\subseteq \om}$.

(1) If $\Delta (A)$ then $\nabla (A)$ is monotonically normal and hereditarily paracompact.

(2) If $\nabla (A)$ is monotonically normal and, whenever $\xbar, \ybar$ in $\nabla(A)$ are compatible then $\xbar \wedge \ybar$ is  in $\nabla(A)$,  then $\Delta (A)$ holds.
\end{thr}
\begin{proof}
For (1) note that if $\Delta (A)$ holds, then by the preceding theorem $\nabla^+ (A)$ is monotonically normal, so its subspace $\nabla (A)$ is monotonically normal.

For (2) assume $\nabla (A)$ is closed under $\wedge$. By \cref{L:Nabla_technical}(iii), for any $\xbar$ and $\ybar$ in $\nabla(A)$ we have that one basic open set in $\nabla (A)$, say $N_{\nabla(A)} (\xbar,f) = N(\xbar,f) \cap \nabla (A)$,  meets another, say $N_{\nabla (A)} (\ybar,g)$, if and only if they both contain $\xbar \wedge \ybar$; and so they meet (in $\nabla (A)$) if and only if the corresponding open sets in $\nabla$, $N(\xbar,f)$ and $N(\ybar,h)$, meet (in $\nabla$). 
Hence if $\nabla (A)$ is monotonically normal then it is monotonically normal in $\nabla$, and thus, by the preceding theorem, $\Delta (A)$ holds.
\end{proof}

From \cref{P:DeltaFI} we deduce:
\begin{exa}\label{Pr:ZFC_Delta}
Let  $\mathrm{FI}$ be the family of finite disjoint unions of increasing partial functions.  Then, in $\mathsf{ZFC}$, we have  $\nabla (FI)$ is monotonically normal, and hereditarily paracompact. 
\end{exa}
A space may be monotonically normal for `trivial' reasons, such as being discrete or metrizable.
Indeed, it is not difficult to check that $\mathrm{INC}$ is a closed and discrete subspace of $\nabla$, and so monotonically normal `trivially'. 
However this is not the case for $\mathrm{FI}$.
For $x$ in $\mathrm{FI}$ let $\mathop{ht} (x)$ be the minimal number of partial functions in a representation of $x$ as a disjoint union of increasing partial functions, and set $\mathrm{FI}_n = \{ x : \mathop{ht}(x)=n\}$. 
Then $\mathrm{FI}_1 = \mathrm{INC}$, the increasing partial functions. 
One can verify that the closure of $\mathrm{FI}_2$ contains $\mathrm{FI}_1$. From \cref{tightness} it follows that every point of $\mathrm{FI}_2$ has uncountable character in $\mathrm{FI}$. Hence $\mathrm{FI}$ is far from being metrizable (or discrete).

\subsection{Another Not Hereditarily Normal Space}

Denote by \emph{$X(\baire, \leq^*)$} the subspace $\baire \cup \{ c_\om \}$ of $\nablaomega$ and write $N(c_\om , f)_X = N(c_\om, f)\cap X(\baire, \leq^*)$, with $f\in \baire$, the neighborhoods around $c_\om$ in $X(\baire, \leq^*)$.

\begin{thr}\label{ThrL(omega_1)xBaireNotHN}
The space $L(\om_1) \times X(\baire, \leq^*)$ is hereditarily normal if and only if $\mathfrak{b}= \om_1$. 
\end{thr}
\begin{proof}
Recall that a subset $L$ of $\om^\om$ ($\baire$ with the product topology) is a \emph{$K$-Luzin} set if it is uncountable and meets every compact of $\baire$ in a countable set, or equivalently, for every $g\in \baire$, the set $\{ f\in L : f\leq^* g\}$ is countable. Observe that any uncountable subspace of $K$-Luzin is $K$-Luzin, hence the existence of a $K$-Luzin set is equivalent to $\mathfrak{b}= \om_1$. 
We prove the equivalence `$L(\om_1) \times X(\baire, \leq^*)$ is hereditarily normal if and only if there is a $K$-Luzin set'.

For the sufficiency, let $p=(\omega_1,c_\om)$ be the top-right corner of the given product. 
In $X' = L(\omega_1) \times X(\om^\om,\le^*) \setminus \{p\}$ the top edge, $T=L(\om_1)\times \{c_\om\} \setminus \{p\}$, and right edge, $R=\{\omega_1\} \times X(\om^\om,\le^*) \setminus \{p\}$, are disjoint closed sets.
Hence, there are disjoint open sets $U$ and $V$ such that $T \subseteq U$ and $R \subseteq V$. For each $\alpha < \omega_1$, pick $f_\alpha$ such that $\{\alpha\} \times N(c_\om,f_\alpha)_X \subseteq U$.
For each $g$ in $\om^\om$ pick countable $C_g \subseteq D(\om_1)$ such that $(L(\om_1)\setminus C_g) \times \{g\} \subseteq  V$.

Let $A=\{f_\alpha : \alpha < \om_1\}$. The choice of the $f_\alpha$'s can be in such way so they are all distinct, so the enumeration of $A$ is injective. We check that $A$ is $K$-Luzin. 
Take any $g$ in $\om^\om$, then for any $\alpha$ not in $C_g$, as $U$ and $V$ are disjoint, $(\alpha,g)$ is not in $\{\alpha\} \times N(c_\om,f_\alpha)_X$, so $f_\alpha \nleq^* g$. Hence, $\{ \alpha \in \om_1 : f_\alpha \le^* g\}$ is contained in $C_g$, and so  is countable.

For the converse, note that $ L(\omega_1) \times X(\om^\om,\le^*)$ is regular and points in $(L(\om_1) \setminus \{ \om_1 \}) \times \baire$ are isolated, and thus this product is hereditarily normal provided: whenever $A \sub T$, $B \sub R$ (where $T$ and $R$ are as above), then there are sets $U, V$ open in $L(\omega_1) \times X(\om^\om,\le^*)$ separating $A$ and $B$.

We show this latter condition holds if there is a $K$-Luzin set. 
Write $A = \{ (\alpha , c_\om) : \alpha \in S \}$, where $S\sub \om_1$. If $A$ is countable, then the result is clear. Hence, suppose $S$ is uncountable. Let $L = \{ f_\alpha : \alpha \in S \} \sub \baire$ be a $K$-Luzin set such that the enumeration is bijective. For every $g\in \baire$, $C_g = \{ \alpha \in S : f_\alpha \leq ^* g\}$ is countable. Hence the open sets $U = \bigcup_{\alpha \in S} \{ \alpha \} \times N(c_\om, f_\alpha)_X$ and $V = \bigcup_{(g, \om_1)\in B} (L(\om_1) \setminus C_g) \times \{ g \}$ separate $A$ and $B$.
\end{proof}

\section{Nabla Products of \texorpdfstring{$A(\kappa)$}{A(kappa)}'s}\label{nabla_of_fort}

\subsection{\texorpdfstring{$\Delta$}{Delta}-like Characterizations of Monotone Normality}
Denote by $D(\kappa)^{\subseteq \om}$ the set of partial functions from $\om$ to $D(\kappa)$,  and $D(\kappa)^{\subset \om}$ for the subset of partial functions with infinite and co-infinite domain. Two elements $x, y \in D(\kappa)^{\subseteq \om}$ \emph{switch}, if $|x\setminus y| = |y\setminus x| = \om$,  and $|\{ n\in \om : x(n), y(n) \in D(\kappa)$ and $x(n) \neq y(n) \}|< \om$. 

\begin{de}
\emph{$\Delta (A(\kappa))$} is the statement: there is $F : D(\kappa)^{\subset \om} \to ([\kappa]^{< \om})^\om$ such that if $x,y \in D(\kappa)^{\subset \om}$ switch, then $(x\setminus y )(n) \in F(y)(n)$ or $(y\setminus x)(n) \in F(x)(n)$ for infinitely many $n\in \om$.
\end{de}

Let $\nabla^* A(\kappa) = \{ \xbar \in \nabla A(\kappa)^\om : x \in D(\kappa)^{\subset \om} \}$. 
For $x$ and $y$ in $D(\kappa)^{\subseteq \omega}$ write $\ybar \preceq \xbar$ if and only if for all but finitely many $n$ in $\mathop{dom} x$ we have $y(n)=x(n)$. 
A basic neighborhood of an $\xbar$ in $\nabla A(\kappa)^\om$ is $N(\xbar, f) = \{ \ybar \in \nabla A(\kappa)^\om : \ybar \preceq \xbar$ and for all but finitely many $n \in \mathop{dom} (y \setminus x)$ we have $(y \setminus x)(n) \notin f(n)\}$, where $f$ is in $([\kappa]^{<\om})^\om$. 

Observe that $\om+1$ is $A(\aleph_0)$ and that all definitions here reduce in the case $\kappa=\aleph_0$ to those in \cref{S:Delta_and_Nabla}. 
The natural analogues of \cref{L:Nabla_technical0} and \cref{L:Nabla_technical} hold. 
Their proofs, and that of the following theorem follow, \textsl{mutatis mutandis}, those for $\Delta (\om^{\subset \om})$ and $\nablaomega$ in \cref{S:Delta_and_Nabla}, and so are omitted.

\begin{thr}\label{CoroA(kappa)MNiffDelta(A(kappa)}
$\Delta (A(\kappa))$ holds if and only if $\nabla A(\kappa)^\om$ is monotonically normal if and only if $\nabla^* A(\kappa)$ is monotonically normal if and only if $\nabla A(\kappa)^\om$ is halvable.
\end{thr}

When can we deduce from $\Delta(A(\kappa))$ that $\nabla A(\kappa)^\om$ is (hereditarily) paracompact? Note that, we can not simply apply \cref{ThrNablaMNimpliesHP}. 
However, for all $\kappa$ we see that $\nabla A(\kappa)^\om$ is homeomorphic to its square, and the second author \cite{gartside1999monotone} has shown that if the square of a space is monotonically normal then all finite powers are monotonically normal and hereditarily paracompact.

\begin{coro}\label{C:DeltaAkappaimpHP}
If $\Delta(A(\kappa))$ holds then $\nabla A(\kappa)^\om$ is hereditarily paracompact, and $\square A(\kappa)^\om$ is paracompact.
\end{coro}


\subsection{Not Hereditarily Normal}

Williams' result in \cite{williams1984box} that under $\mathfrak{d}=\om_1$, countable nabla products of compact spaces of weight no more than $\aleph_1$ are $\om_1$-metrizable, and hence monotonically normal, 
implies, in particular, that consistently $\nabla A(\omega_1)^\om$ is monotonically normal. We now see that this last statement is independent, and  $\om_1$ is the largest cardinal such that $\nabla A(\kappa)^\om$ can be monotonically normal.

\begin{thr}[Roitman \cite{roitman2014box}]\label{ThrNablaA(omega_2)NotHN}
$\nabla A(\om_2)^\om$ is not hereditarily normal.
\end{thr}
\begin{proof}
Since $A(\omega_2)_\delta = L(\om_2)$, this latter space embeds into $\nabla A(\om_2)^\om$. Now, as $\nabla A(\om_2)^\om$  is homeomorphic to its square, \cref{L(omega2)notHereditarilyNorlmal} applies.
\end{proof}
Roitman, in \cite{roitman2014box}, asked: 
is $\nabla A(\om_1)^\om$ consistently non hereditarily normal?
\begin{thr}\label{CoroNablaA(omega_1)NotHN}
If $\mathfrak{b}> \om_1$, then $\nabla A(\om_1)^\om$ is not hereditarily normal.
\end{thr}
\begin{proof}
Since $L(\om_1)=A(\om_1)_\delta$, both spaces $L(\om_1)$ and $\nablaomega$ embed into $\nabla A(\om_1)^\om$, and the latter is homeomorphic to its square. Hence, \cref{ThrL(omega_1)xBaireNotHN} applies.
\end{proof}

\paragraph{Remark} We observe here that a claim of Roitman is incorrect. Theorem~6.1 and Proposition~6.4 in \cite{roitman2011paracompactness} claim: (1)  if $\mathfrak{b}= \mathfrak{d} < \aleph_\om$ and each $X_n$ is compact and has weight $\leq \mathfrak{d}$ then $\nabla_n X_n$ is $\mathfrak{b}$-metrizable (and hence monotonically normal); and
(2) if $\kappa < \mathfrak{b}= \mathfrak{d} < \aleph_\om$ and the nabla product of countably many compact spaces of weight $\kappa$ is $\mathfrak{b}$-metrizable, then the nabla product of countably many compact spaces of weight $\kappa^+$ is $\mathfrak{b}$-metrizable (and hence monotonically normal).
Claim (2) implies claim (1)  by finite induction. 
But both are false. Indeed, the compact spaces $ A(\om_2)$ and $ (\om_2+1)$ have weight $\om_2$, but $\nabla A(\om_2)^\om$ and $\nabla (\om_2 +1)^\om$ are not hereditarily normal as shown in \cref{ThrNablaA(omega_2)NotHN} and \cref{ThrNablaomega2+1NotHN}. Hence, they cannot be $\kappa$-metrizable. In the attempted proof of claim~(2) it is assumed that the nabla product under consideration is $P_{\mathfrak{b}}$, but this is false, in general, when the factors are not first countable.



\section{Nabla Products of  Ordinals}\label{nabla_of_ordinals}

\subsection{\texorpdfstring{$\Delta$}{Delta}-like Characterizations of Monotone Normality}

In this section we uncover a $\Delta$-like combinatorial principle, namely $\Delta(\alpha)$, which characterizes the monotone normality of  a nabla product of  ordinals, $\ordnabla{\alpha}$. 
(For an ordinal $\beta$, write $\mathop{Lim} (\beta)$ for the set of limit ordinals of $\beta$.)

Basic neigborhoods of an $x$ in $\ordnabla{\alpha}$ have the form, $N(x,f) = \{ y : $ for all but finitely many $n$ we have $f(n) \le y(n) \le x(n)$ if $x(n) \in \mathop{Lim}(\alpha)$ and $y(n) = x(n)$ if $x(n)$ isolated$ \}$, where $f$ is in $\alpha^\omega$ and for all but finitely many $n$, if $x(n)$ is a limit then $f(n)<x(n)$. 
Define a partial order $\preceq$ on $\ordnabla{\alpha}$ by saying $y \preceq x$ if for all but finitely many $n$ we have $y(n) \le x(n)$ and if $x(n)$ is isolated then $y(n)=x(n)$. 
Note that $\ordnabla{\alpha}$, the above basic neighborhoods and $\preceq$ satisfy conditions (a) and (b) of \cref{LemmaMNequivalentToHalvable}. 
Hence for $\ordnabla{\alpha}$ to be monotonically normal it suffices to halve the neighbornet $T(x) = \downarrow x = N(x,c_0)$.
Next we state the appropriate notion of `switching' elements in this context and then $\Delta(\alpha)$.
\begin{de}
Let $\alpha$ be any ordinal and $x,y \in \ordnabla{\alpha}$. We say that $x,y$ \emph{switch} if  for infinitely many $n$, $x(n) < y(n) \in \mathop{Lim}(\alpha)$,  for infinitely many $n$, $y(n) < x(n) \in \mathop{Lim}(\alpha)$, and $\{ n\in \om : x(n), y(n)$  are isolated and $x(n) \neq y(n) \}$ is finite.
\end{de}

\begin{de}
\emph{$\Delta(\alpha)$} is the statement: there is $F : \ordnabla{\alpha} \to \alpha^\om$ such that if $x,y\in \ordnabla{\alpha}$ switch, then $y(n) < F(x)(n) < x(n)$ for infinitely many $n$ or $x(n) < F(y)(n) < y(n)$ for infinitely many $n$.
\end{de}

Now we characterize when $\ordnabla{\alpha}$ is monotonically normal.

\begin{prop}\label{PropDelta[alpha]iffNablaalphaMN}
The following are equivalent:
(1) $\Delta(\alpha)$ holds, (2) $\ordnabla{\alpha}$ is halvable, and (3) $\ordnabla{\alpha}$ is monotonically normal.
\end{prop}
\begin{proof}
By the discussion above, it suffices to show the equivalence of (1) and (2${}'$) `the neighbornet $T(x)=\downarrow x$ is halvable'.

For (1) implies (2${}'$), suppose $F$ is a witness of $\Delta(\alpha)$.  
Define $S(x)=N(x, F(x))$.
We check that $S$ halves $T(x) = \downarrow x = N(x, c_0)$. 

Take any $x$ and $y$. 
Suppose $x \notin N(y,c_0)$ and $y \notin N(x,c_0)$. 
Various cases arise, but in all of them we show $S(x)$ and $S(y)$ are disjoint.
If the set $\{ n\in \om : x(n), y(n)$  are isolated and $x(n) \neq y(n) \}$ is infinite, then $S(x)$ and $S(y)$ are trivially disjoint.
Hence, suppose it is finite. 
Then the sets $N_y = \{ n\in \om : x(n) < y(n) \}$ and $N_x = \{ n\in \om : y(n) < x(n)\}$ are both infinite. 
Now, if there are infinitely many $n\in N_y$ such that $y(n)$ is isolated, then $[0, x(n)] \cap \{y(n)\} = \emptyset$, and thus, $S(x)$ and $S(y)$ are disjoint;
and likewise  if there are infinitely many $n\in N_x$ such that $x(n)$ is isolated.
Assume, then, that for all but finitely many $n\in N_y$ and $m\in N_x$, $x(m), y(n) \in \mathop{Lim}(\alpha)$.
That is, $x$ and $y$ switch. By $\Delta(\alpha)$, we have that $S(x) \cap S(y)= \emptyset$. 

For (2${}'$) implies (1), consider the neighbornet $T(x) = N(x, c_0)$. Then, there is a neighborhood assignment $S$ that halves $T$. For $x \in \ordnabla{\alpha}$, let $F(x)\in \alpha^\om$ such that $N(x, F(x)) \sub S(x)$. To see that $F$ satisfies $\Delta(\alpha)$, pick $x,y$  that switch.
This implies $x\notin N(y, c_0)$ and $y\notin N(x, c_0)$, hence by halvability, $N(x, F(x)) \cap N(y, F(y))= \emptyset$. Now it is clear that for infinitely many $n\in \omega$, $y(n) < F(x)(n) < x(n)$ or $x(n) < F(y)(n) < y(n)$.
\end{proof}

As we argued for \cref{C:DeltaAkappaimpHP} we deduce:
\begin{coro}\label{DeltaAlphaimpHP} \ 

(1) If $\Delta(\alpha)$ holds then $\nabla \alpha^\om$ is hereditarily paracompact.

(2) If  $\Delta(\alpha+1)$ holds then $\square (\alpha+1)^\om$ is paracompact.
\end{coro}

It is important to understand the relationship between $\Delta(\alpha)$ and $\Delta(\beta)$, and especially the strength of Roitman's $\Delta=\Delta(\omega+1)$. 
Clearly if $\beta \ge \alpha$ then $\Delta(\beta) \implies \Delta(\alpha)$ (because monotone normality is hereditary and $\ordnabla{\alpha}$ embeds in $\ordnabla{\beta}$). The next two lemmas  give a way to step up.

\begin{lemma}
Let $\alpha$ be an ordinal. Then $\ordnabla{(\alpha.\omega)} = \bigoplus \{ \nabla_n I_n : (I_n)_n \in \mathcal{I}^\om\}$ where $\mathcal{I} = \{[0,\alpha]\} \cup \{(\alpha .n, \alpha.(n+1)] : n\in \om\}$.
\end{lemma}
\begin{proof}
The sets in $\mathcal{I}$ form an open partition of $\alpha. \om$. From `open' we see that each $\nabla_n I_n$ is open  in $\ordnabla{(\alpha.\omega)}$. While from `partition', and the fact that we take every sequence of members of $\mathcal{I}$, we see that the $\nabla_n I_n$ partition $\ordnabla{(\alpha.\omega)}$.
\end{proof}

\begin{lemma}
If $\Delta (\alpha+1)$ holds, then $\Delta (\alpha . \om)$ holds.
\end{lemma}
\begin{proof}
Observe that each $\nabla_n I_n$ from the preceding lemma is homeomorphic to $\ordnabla{(\alpha+1)}$, which is monotonically normal under $\Delta (\alpha+1)$ (\cref{PropDelta[alpha]iffNablaalphaMN}). Since a disjoint sum of  monotonically normal spaces is monotonically normal, we can apply \cref{PropDelta[alpha]iffNablaalphaMN} again to complete the proof. 
\end{proof}

\subsection{Not Hereditarily Normal}
As seen above, under $\mathfrak{d}=\om_1$, we have that $\nabla \alpha^\om$ is monotonically normal for all $\alpha < \om_2$. The next two results provide a sharp contrast.
\begin{thr}\label{ThrNablaomega2+1NotHN}
The space $\nabla (\om_2 +1) ^\om$ is not hereditarily normal.
\end{thr}
\begin{proof}
Let $S = E^{\om_2}_{\om_1}= \{ \alpha \in \om_2 : cf (\alpha) = \om_1 \}$ to $\nabla (\om_2 +1)^\om$. Then $S$ is a stationary subset of $\omega_2$. Note that $\overline{S} = E^{\om_2}_{\om_1} \cup \{ c_{\om_2}\}$, and its $G_\delta$-modification, $\overline{S}_\delta$ are equal. Hence $S$ and $\overline{S}$ both  embed into $\nabla(\om_2+1)^\om$. Since, $\nabla(\om_2+1)^\om$ is homeomorphic to its square, to complete the proof, apply \cref{LemmaSxSNotNormal}.
\end{proof}

\begin{thr}\label{CoroNabla(omega_1+1)NotHN}
If $\mathfrak{b}> \om_1$, then $\nabla (\om_1+1)^\om$ is not hereditarily normal.
\end{thr}
\begin{proof} Let $L$ be the subspace of $\omega_1+1$ consisting of the isolated points along with $\om_1$. 
Then $L_\delta=L$ is homeomorphic to $L(\om_1)$, and so both  $L(\om_1)$ and $\nablaomega$ embed into $\nabla (\om_1+1)^\om$, which is homeomorphic to its square. Hence, \cref{ThrL(omega_1)xBaireNotHN} applies.
\end{proof}


\section{Nabla Products of Metrizable Spaces}\label{nabla_of_metrizable}

\subsection{\texorpdfstring{$\Delta$}{Delta}-like Characterizations of Monotone Normality}

For this section, $\{ (X_n, d_n) : n\in \omega \}$ will be a family of metric spaces. For $x, y\in \nabla_n X_n$ and $f\in \baire$, define $N(x, f) = \nabla_n B_n(x(n) , \rfrac{1}{f(n)})$ and $M(x,f ; y) = \{ n\in \omega : y(n) \notin B_n(x(n) , \rfrac{1}{f(n)}) \}$, where $B_n(a, \varepsilon)$ is $\{a\}$ if $a$ is isolated, and is otherwise the $\varepsilon$-ball in the metric $d_n$. 

We say that $(x,f),(y,g)\in \nabla_n X_n \times \baire$ \emph{switch} if $M(x,f;y)$ and $M(y,g;x)$ are almost disjoint infinite sets. Observe that switching elements $(x,f)$ and $(y,g)$ imply $y \notin N(x,f)$ and $x\notin N(y,g)$.

\begin{de}
Let $\{ (X_n, d_n) : n\in \omega \}$ be a family of metric spaces.
Then $\Delta((X_n,d_n)_n)$ is the statement: there is $F : \nabla_n X_n \times \baire \to \baire$, write $f_x := F(x,f)$, such that if $(x,f), (y,g) \in \mathop{dom} F$ switch, then $\rfrac{1}{f_x (n)} + \rfrac{1}{g_y (n)} < d_n (x(n), y(n))$ holds for infinitely many $n\in \om$. 
\end{de}

The conclusion here, namely $\rfrac{1}{f_x (n)} + \rfrac{1}{g_y (n)} < d_n (x(n), y(n))$, implies that $B_n(x(n), \rfrac{1}{f_x (n)})$ and $B_n(y(n), \rfrac{1}{g_y (n)})$ are disjoint.

\begin{prop}\label{DeltaX_niffnablaX_nMN} \

(1) If $\Delta((X_n, d_n)_n)$ holds then $\nabla_n X_n$ is monotonically normal.

(2) If $\nabla_n X_n$ is monotonically normal, where each $X_n$ is metrizable, then $\Delta((X_n, d_n)_n)$ holds for any choice of compatible metrics, $d_n$ for $X_n$.
\end{prop}
\begin{proof} \ 

\noindent (1) Let $F$ witness  $\Delta((X_n,d_n)_n)$ and define an operator 
$G$ by $G(x, N(x,f))$ is $N(x, \max\{ 2f, f_x \})$. 
We prove that $G$ is a monotone normality operator. First observe that $x\in N(x, \max\{ 2f, f_x \}) \subseteq N(x,f)$. Now, to prove the second property of monotone normality, let $x,y \in \nabla_n X_n, \; f,g \in \baire$ and assume that $y\notin N(x, f)$ and $x \notin N(y, g)$, then we have to prove that $G(x, f) \cap G(y, g) = \emptyset$. There are three cases for the sets $M(x,f ; y)$ and $M(y,g ; x)$:

\begin{itemize}
    \item $M(x,f ; y)$ or $M(y,g ; x)$ is finite: if $M(x,f ; y)$ is finite, then by its definition, $y(n) \in B_n(x(n), \rfrac{1}{f(n)})$ for all but finitely many $n\in \om$. Hence, $y\in N(x, f)$ which is impossible by our assumption.
    \item $M(x,f ; y) \cap M(y,g ; x)$ is infinite: let $Z = M(x,f ; y) \cap M(y,g ; x)$. Then for every $n\in Z$, $y(n) \notin B_n(x(n), \rfrac{1}{f(n)})$ and $x(n) \notin B_n(y(n), \rfrac{1}{g(n)})$. By triangle inequality, $B_n(x(n), \rfrac{1}{2f(n)}) \cap B_n(y(n), \rfrac{1}{2g(n)}) = \emptyset$, for $n\in Z$. Thus, $G(x, f) \cap G(y, g) = \emptyset$.
    \item $M(x,f ; y)$, $M(y,g ; x)$ are infinite almost disjoint sets: this means that $(x,f),(y,g)$ switch. By $\Delta((X_n, d_n)_n)$, for infinitely many $n\in \om$, $\rfrac{1}{f_x (n)} + \rfrac{1}{g_y (n)} < d_n (x(n), y(n))$. That is, for infinitely many $n\in \om$, the sets $B_n(x(n), \rfrac{1}{f_x (n)})$ and $B_n(y(n), \rfrac{1}{g_y (n)})$ are disjoint, which implies that $G(x, f) \cap G(y, g) = \emptyset$.
\end{itemize}
This concludes the proof of (1).

\medskip

(2) Now, assume that $\nabla_n X_n $ is monotonically normal with operator $G$. For each $n$ let $d_n$ be the given compatible metric on $X_n$.
Define $F: \nabla_n X_n \times \baire \to \baire$ by $F(x, f) = f_x$ such that $N(x, f_x) \sub G(x, N(x,f))$. 

We prove that $F$ witnesses $\Delta((X_n, d_n)_n)$. Choose switching elements $(x, f)$, $(y,g) \ \in \mathop{dom} F$. Then $y\notin N(x,f)$ and $x\notin N(y,g)$. Since $G$ is a monotone normality operator, $G(x,N(x,f)) \cap G(y,N(y,g)) = \emptyset$, which implies $N(x,f_x) \cap N(y,g_y) = \emptyset$. Hence, $B_n(x(n), \rfrac{1}{f_x (n)}) \cap B_n(y(n), \rfrac{1}{g_y (n)}) = \emptyset$ for infinitely many $n\in \om$, and for these $n$'s we have the inequality $\rfrac{1}{f_x (n)} + \rfrac{1}{g_y (n)} < d_n (x(n), y(n))$, as desired.
\end{proof}

For a sequence of metrizable spaces $(X_n)_n$, define $\Delta((X_n)_n)$ to mean  `the statement $\Delta((X_n,d_n)_n)$ holds for some choice of compatible metrics $d_n$'. It follows from the preceding result that $\Delta((X_n)_n)$ is equivalent to `$\Delta((X_n,d_n)_n)$ holds for \emph{any} choice of compatible metrics $d_n$'. 
Further, for a class $\mathcal{C}$ of spaces,  $\Delta(\mathcal{C})$ means `$\Delta((X_n)_n)$ holds for any sequence $(X_n)_n$ of spaces from $\mathcal{C}$'.

Write $\mathcal{M}$ for the class of all metrizable spaces, $\mathcal{M}(\kappa)$ for the class of all metrizable spaces of cardinality $\le \kappa$, and $\mathcal{SM}$ for the class of separable metrizable spaces.

\subsection{Consistency of \texorpdfstring{$\Delta$}{Delta}-like Principles}
\begin{prop} 
If $\mathfrak{b}= \mathfrak{d}$ then $\Delta(\mathcal{M})$ holds, 
if $\mathfrak{d} = \mathfrak{c}$ then $\Delta(\mathcal{M}(\mathfrak{c}))$ holds, and
if $\mathsf{MH}$ holds then  $\Delta(\mathcal{SM})$ holds.
\end{prop}
\begin{proof} We deal with each case in turn. 

\smallskip

\noindent \textbf{For $\mathfrak{b} = \mathfrak{d}$:}  Let $\{ f_\alpha : \alpha < \mathfrak{d} \}$ be a scale  such that $f_\alpha$ $\leq^*$-dominates $\{ 2f_\beta : \beta < \alpha \}$. Let $\down{f_\alpha} = \{ f\in \om : f \leq^* f_\alpha \}$ and define $F: \nabla_n X_n\times \baire \to \baire$ as $F(x,f) = 2 f_\alpha$ if and only if $\alpha$ is the least such that $f \in \; \down{f_\alpha}$.

Pick switching elements $(x,f),(y,g) \in \nabla_n X_n\times \baire$. We may assume that $f\in \; \down{f_\alpha}$ and $g\in \; \down{f_\beta}$, for minimum $\beta, \alpha$ and $\beta \leq \alpha$. Then, $g \leq ^* f_\beta \leq^* 2f_\beta \leq^* f_\alpha$. Now, if $n\in M(y,g;x)$ then $x(n) \notin B_n(y(n), g(n))$, that is, $\rfrac{1}{g(n)} < d_n(y(n), x(n))$. Consequently, $\rfrac{1}{2g(n)} + \rfrac{1}{2 f_\alpha (n)} < d_n(y(n), x(n))$ which implies $\rfrac{1}{2 f_\beta(n)} + \rfrac{1}{2f_\alpha (n)} < d_n(y(n), x(n))$, as desired.
\medskip

\noindent \textbf{For $\mathfrak{d} = \mathfrak{c}$:} If $X_n$ is metrizable of size no more than $\mathfrak{c}$, for $n\in \om$, then we have $|\nabla_n X_n| = \mathfrak{d}$. 
Enumerate $\nabla_n X_n \times \baire = \{ (x_\alpha, f_\alpha) :  \alpha < \mathfrak{d} \}$. Fix $\alpha < \mathfrak{d}$ and suppose $F$ is constructed satisfying $\Delta((X_n)_n)$ on $\{ (x_\beta, f_\beta) : \beta < \alpha \}$ and $F(x_\beta, f_\beta) \geq^* 2f_\beta$, for every $\beta < \alpha$. 
The sets $\mathcal{F} = \{ 2f_\beta : \beta < \alpha \}$ and $\mathcal{A} = \{ M(x_\beta, f_\beta ; x_\alpha) : \beta < \alpha \}$ have size less than $\mathfrak{d}$. \cref{LemmaDiagonilizingFunctions} applies, there is $f'_\alpha \in \baire$ such that $2f_\beta \rest M(x_\beta, f_\beta; x_\alpha)\ngtr^* f'_\alpha$, for $\beta < \alpha$. 
Define $F(x_\alpha, f_\alpha)= 2 \max \{f_\alpha, f'_\alpha\}$. This construction completes $F$ on $\nabla_n X_n \times \baire$. To see that $F$ witness $\Delta((X_n)_n)$, pick switching elements $(x_\beta, f_\beta), (x_\alpha, f_\alpha)$, and suppose $\beta < \alpha$.  
Let  $M = \{ n\in  M(x_\beta, f_\beta; x_\alpha) : f'_\alpha (n) > 2f_\beta \}$, which is infinite. Hence, for $n\in M$, $x_\alpha(n) \notin B_n(x_\beta(n), \rfrac{1}{f_\beta(n)})$.
It is clear that for $n\in M$, $\rfrac{1}{F(x_\beta, f_\beta)(n)} + \rfrac{1}{F(x_\alpha, f_\alpha)(n)} < d_n (x_\beta (n) , x_\alpha (n))$.  

\medskip

\noindent \textbf{For $\mathsf{MH}$:} Every separable metrizable space embeds into the Hilbert cube $[0,1]^\omega$, which is isomorphic to a subset of $H(\omega_1)$. Hence if $(X_n)_n$ is a sequence of separable metrizable spaces, then we can suppose $\nabla_n X_n \subseteq H(\omega_1)$.

 Let $H_\alpha$ be as in $\mathsf{MH}$ (\cref{MH}) and $f_\alpha$ be a witness that $H_\alpha \cap \baire$ is not dominating. We may assume that $H_\alpha \subseteq H_{\alpha +1}$ and that $f_\alpha \in H_{\alpha +1}$. Define $F: \nabla_n X_n\times \baire \to \baire$ as $F(x,f) = 2f_\alpha$ if and only if $\alpha$ is the least such that $(x,f) \in H_\alpha$. Choose switching elements $(x,f), (y,g) \in \mathop{dom} F$. Then $(x,f)\in H_\beta, (y,g) \in H_\alpha$, for minimum $\beta \leq \alpha$. Since the $H_\alpha$'s are elementary submodels, $2g, 2f, M(x,f; y)$ and $M(y,g;f)$ are in $H_\alpha$. Also, for any $h\in H_\alpha \cap \baire$ and $a\in H_\alpha \cap [\om]^\om$,  $h \rest a\ngtr^* f_\alpha$. Hence, $2g\rest M(y,g;x) \ngtr^* f_\alpha$ which implies that there is an infinite set $M \subseteq M(y,g;x)$ such that for $n\in M$, $2g(n) \leq f_\alpha(n)$ and $x(n) \notin B_n (y(n), \rfrac{1}{g(n)})$. As a consequence, for every $n\in M$, $\rfrac{1}{f_\alpha(n)} + \rfrac{1}{2g(n)} < d_n (x(n) , y (n))$. We conclude that for $n\in M$, $\rfrac{1}{F(x, f)(n)} + \rfrac{1}{F(y, g)(n)} < d_n (x(n) , y (n))$.  
\end{proof}


\section{Open Problems} 

The most basic open question, related to this paper, is that of Roitman:
\begin{ques}
Is $\neg \Delta$ consistent? Or is $\Delta$ true in $\mathsf{ZFC}$?
\end{ques}
Suppose $\Delta=\Delta(\om+1)$ were true in $\mathsf{ZFC}$, so $\nablaomega$ is monotonically normal. Then it seems implausible to the authors that $\nabla (\om.\om+1)^\om$ would not also be monotonically normal, in other words $\Delta(\om.\om+1)$ could be consistently false. 
If that is correct then, in $\mathsf{ZFC}$, it should be possible to deduce $\Delta(\om.\om+1)$ from $\Delta(\om+1)$. 
\begin{prob}
Show, in $\mathsf{ZFC}$, that $\Delta(\om+1) \implies \Delta(\om.\om+1)$.
\end{prob}

We know $\nabla A(\om_1)^\om$ is monotonically normal under $\mathfrak{d}=\om_1$; and that $\nabla A(\om_1)^\om$  monotonically normal implies $\mathfrak{b}=\om_1$. This leaves a gap.
\begin{prob}
What is the consistency strength of `$\nabla A(\om_1)^\om$ is monotonically normal'. Does $\mathfrak{b}=\om_1$ suffice?
\end{prob}

We have seen that Roitman's $\Delta$ is equivalent to $\nablaomega$ being monotonically normal. Hence $\Delta$ is a sufficient condition for $\nablaomega$ to be hereditarily paracompact. A natural question then is it necessary? 
But a more fundamental problem is to find necessary \emph{combinatorial} conditions for $\nablaomega$ to be paracompact. That would open a path to showing the independence of the box product problem.

\begin{prob}
Find combinatorial properties (in the style of $\Delta$) implied by `$\nablaomega$ is paracompact'.
\end{prob}


\bibliography{arxiv-mn_nabla}{}
\bibliographystyle{plain}

\end{document}